\documentclass{article}
\usepackage{amsmath}
\usepackage{amssymb}
\usepackage{amsthm}
\usepackage{epsfig}

\newtheorem{theorem}{Theorem} [section]

\newtheorem{conjecture}[theorem]{Conjecture}

\newtheorem{lemma}[theorem]{Lemma}

\newcommand{\dist}{{\text{\upshape dist}}}
\newcommand{\rad}{\text{\upshape rad}}
\begin{document}

\title{Maximum matchings in regular graphs}
  \author{
   Dong Ye \thanks{
Department of Mathematical Sciences and Center for Computational 
Sciences,
Middle Tennessee State University,
  Murfreesboro, TN 37132;
Email: dong.ye@mtsu.edu. Partially supported by a grant from Simons Foundation (no. 369519).}}

\date{}

\maketitle
\begin{abstract}
It was conjectured by Mkrtchyan, Petrosyan, and Vardanyan that
every graph $G$ with
$\Delta(G)-\delta(G) \le 1$ has a maximum matching $M$ such that any two 
$M$-unsaturated vertices do not share a
neighbor. In this note, we confirm the conjecture for all $k$-regular simple graphs and also $k$-regular 
multigraphs with $k\le 4$.
\end{abstract}

\section{Introduction}

Graphs considered in this paper may have multi-edges, but no loops. A graph without multi-edges is called a {\em simple} graph. A 
 {\em matching} $M$of  a graph $G$ is a set of independent edges. A vertex
 is {\em $M$-saturated} 
if it is incident with an edge of $M$, and {\em $M$-unsaturated} otherwise. 
A matching $M$ is said to be {\em maximum} if for any other matching $M'$, 
$|M|\ge |M'|$.
A matching $M$ is {\em perfect} if it covers 
all vertices of $G$.  If $G$ has a perfect matching, the every maximum matching 
is a perfect matching. 
The maximum and minimum degrees of a graph $G$ 
are denoted by $\Delta(G)$ and $\delta(G)$, respectively. 
Mkrtchyan, Petrosyan and Vardanyan \cite{MPV10, MPV13} made the following conjecture.

\begin{conjecture} [Mkrtchyan et. al. \cite{MPV10, MPV13}]\label{conj}
Let $G$ be a graph with $\Delta(G)-\delta(G)\le 1$. Then $G$ contains a maximum matching 
$M$ such that any two $M$-unsaturated vertices do not share a neighbor. 
\end{conjecture}

This conjecture is verified for subcubic graphs (i.e. $\Delta(G)= 3$) by Mkrtchyan, Petrosyan and Vardanyan 
\cite{MPV10}. Later,
Picouleau \cite{CP} find a counterexample to the conjecture, which is a bipartite simple graph
with
$\delta(G)=4$ and $\Delta(G)=5$. Petrosyan~\cite{P} constructs counterexamples to the conjecture  for all $k$-regular graphs with $k\ge 7$ and
for graphs $G$ with $\Delta(G)-\delta (G)=1$ and 
$\Delta(G)\ge 4$.
Note that, most of counterexamples of Conjecture~\ref{conj} for graphs $G$ with $\Delta(G)-\delta(G)=1$ are simple, but
all $k$-regular graphs with $k\ge 7$ given by Petrosyan \cite{P} have multi-edges.
As affirmative answer to Conjecture~\ref{conj} is known only for graphs with $\Delta(G)\le 3$,
Mkrtchyan et. al \cite{MPV10} asked whether the conjecture holds for any $k$-regular graphs with $k\ge 4$.

In this note, we consider the 
conjecture for both $k$-regular simple graphs and $k$-regular graphs with multi-edges. 
First we show that
Conjecture~\ref{conj} does hold for all $k$-regular simple graphs. 

\begin{theorem}\label{thm:simple}
Let $G$ be a $k$-regular simple graph. Then $G$ has a maximum matching $M$
such that any two $M$-unsaturated vertices do not share a neighbor.
\end{theorem}

Further, we show that Conjecture~\ref{conj} holds for $k$-regular graphs with multi-edges for $k\le 4$.

\begin{theorem}\label{thm:4-regular}
Let $G$ be a $k$-regular graph with $k\le 4$. Then $G$ has a maximum matching $M$
such that any two $M$-unsaturated vertices do not share a neighbor. 
\end{theorem}

Our results together with examples given by Petrosyan~\cite{P} leave Conjecture~\ref{conj} unkown for 5 and 6-regular 
graphs with multi-edges.

\section{Preliminaries}
Let $G$ be a graph and $v$ be a vertex of $G$. The {\em neighborhood} of $v$ is 
set of all vertices adjacent to $v$, denoted by 
$N(v)$. The degree of $v$ is $d_G(v)=|N(v)|$. 
If there is no confusion, we use $d(v)$ instead. For $X\subseteq V(G)$, 
let $\delta(X):=\min\{d(v)|v\in X\}$ and $\Delta(X):=\max\{d(v)| v\in X\}$. The 
neighborhood of $X$ is defined as $N(X):=\{y|y\mbox{ is a neighbor of a vertex } x\in X\}$.
For two subsets $X_1$ and $X_2$ of $V(G)$, use $[X_1, X_2]$ to denote the all edges with one 
endvertex in $X_1$ and another endvertex in $X_2$.
For two subgraphs $G_1$ and $G_2$ of $G$,
the symmetric difference of $G_1\oplus G_2$ is defined as a subgraph 
with vertex set $V(G_1)\cup V(G_2)$ and edge set $(E(G_1)\cup E(G_2))\backslash 
(E(G_1)\cap E(G_2))$. 

A matching of a graph $G$ is a {\em near-perfect matching} if it covers all 
vertices except one. If a graph $G$ has a near perfect matching, then $G$ has odd
number of vertices. A graph 
is {\em factor-critical} if, for any vertex $v$, the subgraph $G\backslash\{v\}$ has a perfect matching.
Every maximum matching of a factor-critical graph is a near-perfect matching. 

Let $D$ be the set of all vertices of a graph
$G$ which are not covered by 
at least one maximum matching, and $A$, the set of all vertices in $V(G)-D $ adjacent
to at least one vertex in $D $. Denote $C =V(G)-A -D$. The graph induced by 
all vertices in
$D$ (resp. $A$ and $C$) is denoted by $G[D]$ (resp. $G[A]$ and $G[C]$). The 
following theorem characterizes the structures of maximum matchings of 
graphs, which is due to Gallai \cite{G64} and Edmonds \cite{E65}.

\begin{theorem}[Gallai-Edmonds Structure Theorem, Theorem 3.2.1 in \cite{LP}]
Let $G$ be a graph, and $A$, $D$ and $C$ are defined as above. Then:\\
(1) the components of the subgraph induced by $D$ are factor-critical;\\
(2) the subgraph induced by $C$ has a perfect matching;\\
(3) if $M$ is a maximum matching of $G$, it contains a near-perfect matching
of each component of $G[D]$, a perfect matching of $G[C]$ and matches all vertices
of $A$ with vertices in distinct components of $G[D]$.
\end{theorem}

Contract every component of $G[D]$ to a vertex and let $B$ be the set of all these 
vertices. Then the graph obtained from $G\backslash C$ by contracting all components of 
$G[D]$ to a vertex and deleting all generated loops is a bipartite graph, denoted by $G(A,B)$.
Because every component of $G[D]$ is factor-critical, a maximum matching of $G(A,B)$ 
is corresponding to a maximum matching 
of $G$, and vice versa. Before processing to prove our main results, we need some 
results for maximum matchings of bipartite graphs $G(A,B)$. 

\begin{theorem}[Hall's Theorem, Theorem 1.13 in \cite{LP}]
\label{thm:Hall}
Let $G(A, B)$ be a bipartite graph. If $|N(S)|\ge |S|$ for any $S\subseteq A$, then 
$G$ has a matching $M$ covering all vertices of $A$. 
\end{theorem}

The following techincal lemma is needed in proof of our main results.

\begin{lemma}\label{lemma:case4}
Let $G(A,B)$ be a bipartite graph such that every maximum matching of $G(A,B)$ covers
all vertices of $A$. 
Let $W\subseteq B$ such that $\delta(W)\ge \Delta(A)$.
Then $G(A,B)$ has a maximum matching $M$ covering all vertices of $W$.
\end{lemma}
\begin{proof}
Let $M$ be a maximum matching of $G(A,B)$ such that the number of vertices of $W$ 
covered by $M$ is maximum.
If $M$ covers all vertices of $W$, the lemma follows.
So assume that there exists an $M$-unsaturated vertex  $x\in W$. 

For any $U\subseteq W$, we have $\delta(U)\ge  \delta(W)$ and $N(U)\subset A$. Further,
\[\delta(W)|U|\le \delta(U) |U|\le  |[U, N(U)]| \le \sum_{v\in N(U)} d(v) 
\le \Delta(A)|N(U)|.\] It follows
that $|N(U)|\ge |U|$ because $\delta(W)\ge \Delta(A)$. 
By applying Hall's 
Theorem on the subgraph induced by $W$ and $N(W)$, it follows that
$G$ has a matching $M'$ covering all vertices of $W$.  
 
Let $M\oplus M'$ be the symmetric difference of $M$ and $M'$. Every
component of $M\oplus M'$ is either a path or a cycle. Since $x$ is not covered by $M$ but is covered by $M'$, 
it follows that $x$
is an end-vertex of some path-component $P$ of $M\oplus M'$. Let $y$
be another end-vertex of $P$. Note that every vertex of $A$ is covered by an edge
of $M$ and every vertex of $W$ is covered by an edge of $M'$. 
So $y\in B\backslash W$.

Then let $M''=M\oplus P$. Then $M''$ is a maximum matching of $G$ which
covers $x$ and all vertices covered by $M$ except $y$. Note that $y\in B\backslash W$
and $x\in W$. Hence $M''$ covers more vertices of $W$ than $M$, a contradiction to 
the maximality of the number of vertices of $W$ covered by $M$. 
This completes the proof.
\end{proof}

\section{Proof of main results}
 

Let $G$ be a $k$-regular graph. Without loss of generality, assume that $G$ is connected. 
Otherwise, we consider each connected component of $G$. Let $M$ be a maximum matching of
$G$. 
If $|M|\ge (|V(G)|-1)/2$, then $G$ has at most one $M$-unsaturated
vertex. Theorem~\ref{thm:simple} and Theorem~\ref{thm:4-regular} 
hold automatically. So in the following, assume 
$|M|<(|V(G)|-1)/2$. So $k\ge 3$.

By Gallai-Edmonds Structure Theorem, $V(G)$ can be partitioned into
three parts $C$, $A$ and $D$ such that every maximum matching of $G$
matches all vertices of $A$ with vertices in distinct components of $G[D]$.
Let $c(D)$ be the number of components
of $G[D]$. Then $|M|=|C|/2+(|D|-c(D))/2+|A|$ by Gallai-Edmonds Structure
Theorem.
So \[|C|/2+(|D|-c(D))/2+|A|=|M|<(|V(G)|-1)/2=(|C|+|A|+|D|-1)/2.\]
It follows that $c(D)\ge 2+|A|$. 

Let $Q_1, Q_2, ...,Q_{c(D)}$ be all components of $G[D]$. Let $[Q_i,A]$ (resp. $[D,A]$)
be the set of all edges joining a vertex of $Q_i$ (resp. $D$) and a vertex of $A$. 
Note that \[\sum_{i=1}^{c(D)}|[Q_i,A]|=|[D,A]|\le k|A|\]
because $G$ is $k$-regular. Let $G/Q_i$ be the graph 
obtained by contracting $Q_i$ and deleting all loops. Note that $Q_i$ is 
factor-critical and hence has odd number of vertices, and 
$G/Q_i$ has even number of vertices of odd degree. So the degree of
the new vertex of $G/Q_i$ corresponding to $Q_i$ has the same parity as $k$. 
It follows that \[|[Q_i,A]|\equiv k \pmod 2.\] 
In the following, we always assume that $|[Q_i, A]|\ge |[Q_j, A]|$ for $i\le j$. Then 
there exists an integer $t<|A|$ such that $|[Q_i,A]|< k$ for any $i\ge t$. A vertex $v$
of $Q_i $ is {\bf good} if all neighbors of $v$ are contained in $Q_i$.

\medskip

\noindent{\bf Proof of Theorem \ref{thm:simple}.}  Since $G$ is a simple graph, for each
$Q_i$, we have 
\[\frac{k|V(Q_i)|-|[Q_i,A]|}{2} \le {|V(Q_i)|\choose 2}.\]
Note that $|[Q_i,A]|< k$ if $i\ge t$. It follows that $|V(Q_i)|>k$ for $i\ge t$. So at least one vertex of $Q_i$ with $i\ge t$ has
no neighbors in $A$. Hence every component of $Q_i$ with
$i\ge t$ contains a good vertex.  Choose a good vertex $v_i$ from each $Q_i$ with $i\ge t$ and let $X$ be the set of all chosen good vertices 
$v_i$. Then any two vertices of $X$ do not share a neighbor because $Q_i\cap Q_j=\emptyset$ if $i\ne j$.


Contract all components $Q_i$ into a vertex $q_i$, and let  $B=\{q_i|i=1,2,..., c(D)\}$. Let $G(A,B)$ be the bipartite graph with 
bipartition $A$ and $B$, and all edges in $[D,A]$ of $G$.  
Let $W:=\{q_i| q_i\in B\mbox{ and }d_H(q_i)\ge k\}$, the set of vertices corresponding to such 
$Q_i$ with $|[Q_i,A]|\ge k$ (i.e, $i< t$). By Gallai-Edmonds Structure Theorem, every maximum matching of $G(A,B)$
covers all vertices of $A$.  By Lemma~\ref{lemma:case4}, 
$G(A,B)$ has a maximum matching $M$ covering all vertices of $W$ and all vertices of $A$.
In the graph $G$, $M$ is a matching which covers all vertices of $A$, and a vertex from every $Q_i$ with $i<t$, and a vertex 
from some $Q_j$ with $j\ge t$. 
For each $Q_i$, let $M_i$ be a near-perfect matching covering all vertices except the vertex covered by $M$ or the good vertex $v_i$ if the component $Q_i$ has no vertex covered by $M$.

By Gallai-Edmonds Structure Theorem, $G[C]$ has a perfect matching $M_C$. Let $M'$ be the union of $M$, $M_C$ and all $M_i$'s. Then 
$M'$
is a maximum matching of $G$. So all $M'$-unsaturated vertices belong to $X$. So any two $M'$-unsaturated vertices do not share a neighbor. This completes the proof.
\qed\medskip

Now we are going to prove Theorem~\ref{thm:4-regular}.\medskip

\noindent{\bf Proof of Theorem \ref{thm:4-regular}.} Let $G$ be a $k$-regular graph with multi-edges and 
$k\le 4$.  
Note that, $|[Q_i,A]|\ge k$ if $i< t$ and $|[Q_i,A]|<k$ if $i\ge t$.  
Since $|[Q_i,A]|\equiv k \pmod 2$, it follows that $|[Q_{i}, A]| =k-2$ for $i\ge t$. 
Hence $Q_i$ for $i\ge t$ 
is not a singleton. Further,
$Q_i$ with $i\ge t$ has at least three vertices
because $Q_i$ is factor-critical.  So every component
$Q_i$ for $i\ge t$ has a good vertex $v_i$. 

A similar argument as above shows that $G$ has a maximum matching $M'$ which covers all vertices of
$G$ except some good vertices from different components $Q_i$'s of $D$. Since any two good vertices from 
different $Q_i$ and $Q_j$ do not share a neighbor, the theorem follows. \qed


\begin{thebibliography}{}
\parskip=-0.1cm
\bibitem{E65} J. Edmonds, Paths, trees and flowers, Canad. J. Math 17 (1965) 449--467.

\bibitem{G64} T. Gallai, Maximale systeme unabh\"{a}ngiger kanten, Magyar Tud. 
Akad. Mat. Kutat\'o Int. K\"{o}zl. 9 (1964) 401--413.

\bibitem{LP} L. Lov\'asz and M.D. Plummer, Matching Theory, North Holland, Amsterdan, 1986.

\bibitem{MPV10} V.V. Mkrtchyan, S.S. Petrosyan and G.N. Vardanyan, On disjoint 
matchings in cubic
graphs, Discrete Math. 310 (2010) 1588--1613.

\bibitem{MPV13} V.V. Mkrtchyan, S.S. Petrosyan and G.N. Vardanyan, Corrigendum to 
``On disjoint matchings in cubic graphs" [Discrete Math. 310 (2010) 1588--1613], 
Discrete Math. 313 (2013) 2381.

\bibitem{P} P.A. Petrosyan, On maximum matchings in almost regular graphs, Discrete Math. 318 (2014) 58--61.

\bibitem{CP} C. Picouleau, A note on a conjecture on maximum matching in almost regular graphs,
Discrete Math. 310 (2010) 3646--3647. 
\end{thebibliography}
\end{document}